\documentclass[american, letterpaper]{article}

\usepackage[margin=1in]{geometry}
\makeatletter
\def\thm@space@setup{%
\thm@preskip=1em \thm@postskip=0pt
}
\makeatother

\usepackage{amsmath, amsthm, amssymb, bbold, mathtools}
\usepackage{graphicx}
\usepackage{acronym}
\usepackage{multirow}
\usepackage{import}
\usepackage{url}

\usepackage[
backend=bibtex,
style=authoryear,
natbib=true,
url=false, 
doi=false,
eprint=false,
dashed=false,
maxbibnames=99
]{biblatex}
\usepackage{hyperref}
\usepackage[]{hyperref}
\hypersetup{
  colorlinks=false,
}

\usepackage[ruled,vlined]{algorithm2e}

\newtheorem{proposition}{Proposition}
\newtheorem{lemma}{Lemma}

\newtheorem{theorem}{Theorem}

\acrodef{qop}[QOP]{Quadratic Optimization Problem}
\acrodef{socp}[SOCP]{Second-Order Cone Problem}
\acrodef{sdp}[SDO]{Semidefinite Optimization}
\acrodef{mio}[MIO]{Mixed Integer Optimization}
\acrodef{cio}[CIO]{Convex Integer Optimization}
\acrodef{kkt}[KKT]{Karush-Kuhn-Tucker}
\acrodef{amp}[AMP]{Approximate Message-Passing}
\acrodef{tsp}[TSP]{Travelling Salesman Problem}
\acrodef{iid}[i.i.d.]{independent identically distributed}
\acrodef{snr}[SNR]{signal-to-noise ratio}

\usepackage{tikz}
\usepackage{pgfplots}
\usetikzlibrary{shapes,arrows}
\tikzstyle{block} = [draw, rectangle, minimum height=3em, minimum width=6em]
\tikzstyle{input} = [coordinate]
\tikzstyle{output} = [coordinate]
\usetikzlibrary{external}
\newcommand{\norm}[1]{\left\|#1\right\|}
\newcommand{\abs}[1]{\left|#1\right|}
\newcommand{\one}[0]{\mathbb{1}}

\newcommand{\eye}[1]{\mathbb{I}_{#1}}
\newcommand{\set}[2]{\left\{ #1\ : \ #2 \right\}}
\newcommand{\tpose}{^\top}
\newcommand{\defn}[0]{:=}
\providecommand{\e}[1]{\ensuremath{\cdot 10^{#1}}}
\newcommand{\mc}{\mathcal}
\newcommand{\mb}{\mathbb}

\def\N{\mathrm{N}}
\def\Re{\mathrm{R}}
\renewcommand{\S}{\mathrm{S}}
\def\st{\mathrm{s.t.}}

\DeclareMathOperator{\supp}{supp}

\providecommand{\keywords}[1]{\textbf{\textit{Keywords: }} #1}

\usepackage[affil-it]{authblk}

\title{Sparse Hierarchical Regression with Polynomials}

\author[1]{Dimitris Bertsimas\thanks{\href{mailto:dbertsim@mit.edu}{dbertsim@mit.edu}}}
\author[1]{Bart \mbox{Van Parys}\thanks{\href{mailto:vanparys@mit.edu}{vanparys@mit.edu}}}

\affil[1]{Operations Research Center, Massachusetts Institute of Technology}

\date{}

\begin{document}

\maketitle

\begin{abstract}
  We present a novel method for exact hierarchical sparse polynomial regression. Our regressor is that degree $r$ polynomial which depends on at most $k$ inputs, counting at most $\ell$ monomial terms, which minimizes the sum of the squares of its prediction errors. The previous hierarchical sparse specification aligns well with modern big data settings where many inputs are not relevant for prediction purposes and the functional complexity of the regressor needs to be controlled as to avoid overfitting. We present a two-step approach to this hierarchical sparse regression problem. First, we discard  irrelevant inputs using an extremely fast input ranking heuristic. Secondly, we take advantage of modern cutting plane methods for integer optimization to solve our resulting reduced hierarchical $(k, \ell)$-sparse problem exactly. The ability of our method to identify all $k$ relevant inputs and all $\ell$ monomial terms is shown empirically to experience a phase transition. Crucially, the same transition also presents itself in our ability to reject all irrelevant features and monomials as well. In the regime where our method is statistically powerful, its computational complexity is interestingly on par with \texttt{Lasso} based heuristics. The presented work fills a void in terms of a lack of powerful disciplined nonlinear sparse regression methods in high-dimensional settings. Our method is shown empirically to scale to regression problems with $n\approx 10,000$ observations for input dimension $p\approx 1,000$.
  
\end{abstract}

\keywords{Nonlinear Regression, Sparse Regression, Integer Optimization, Polynomial Learning}

\section{Introduction}
\label{sec:introduction}

We consider the problem of high-dimensional nonlinear regression. Given input $X=(x_1, \dots, x_n) \in \Re^{n\times p}$ and response data $Y = (y_1, \dots, y_n) \in \Re^n$, we set out to find an unknown underlying nonlinear relationship 
\begin{equation*}
  y_t=g(x_t) + e_t, \quad \forall t \in [n],
\end{equation*}
where $E\defn(e_1, \dots, e_n)$ in $\Re^n$ is the error term. High-dimensional regression is a problem at the core of  machine learning, statistics and signal processing. It is clear that if we are to carry any hope of success, some structure on the nature of the unknown nonlinear relationship $g$ between input and response data must be assumed. Classical statistical learning theory \citep{vapnik2013nature} indeed requires the complexity of the set of considered functions to be bounded in some way.
Polynomial regression has a long history \citep{smith1918standard} and is mentioned in almost any standard work on machine learning. We will consider in this paper all nonlinear relationships in the form of polynomials in $p$ variables of total degree at most $r$. We denote with $\mc P$ the set of polynomials of total degree $r$. Typically the polynomial which best explains data is defined as the minimizer to the abstract optimization problem
\begin{equation}
  \label{eq:abstract_regression}
  \min_{g \in \mc P} ~\small\frac{1}{2} \sum_{t\in[n]} \norm{y_t - g(x_t)}^2 + \frac{1}{2\gamma}\norm{g}^2,
\end{equation}
over polynomial functions $g$ in $\mc P$. The squared norm $\norm{g}^2$ of a polynomial $g$ is taken here to mean the sum of squares of its coefficients. The best polynomial in formulation \eqref{eq:abstract_regression} minimizes a weighted combination of the sum of its squared prediction errors and its coefficient vector. This latter Ridge regularization \citep{tikhonov1943stability, hoerl1970ridge} term stabilizes its solution and avoids overfitting. An alternative interpretation of the regularization term as a precaution against errors in the input data matrix $X$ has been given for instance in  \citep{bertsimas2017characterization}. Nevertheless, the value of the hyperparameter $\gamma$ must in practice be estimated based on historical data using for instance cross validation. 

By allowing for nonlinear dependence between input and response data polynomial regression can discover far more complex relationships than standard linear regression. For a sufficiently large degree $r$ in fact, any continuous functional dependence can be discovered up to arbitrary precision \citep{stone1948generalized}. The previous observation leads to the fact that polynomial regression is a hybrid between parametric and nonparametric regression. Depending on the degree $r$ of the polynomials considered, it falls between textbook parametric regression $(r=1)$ which assumes the functional dependence $g$ between input and response data to be linear and completely nonparametric regression ($r\to \infty$) where nothing beyond continuity of the dependence $g$ between input and response data is assumed. Although nonparametric approaches are very general and can unveil potentially any continuous relationship between input and response data, they nonetheless seriously lack in statistical power. Indeed, nonparametric methods such as kernel density estimation \citep{turlach1993bandwidth} need a huge number of samples in order to return statistically meaningful predictions. This curse of dimensionality is especially harmful in high-dimensional settings $p \gg n$ commonly found in modern data sets.

The polynomial regressor $g$ in the regression problem \eqref{eq:abstract_regression} is a sum of at most $f \defn \binom{p+r}{r}$ monomial features. A seminal result due to \citet{vapnik1998support} states that the high-dimensionality of the unconstrained regression problem \eqref{eq:abstract_regression} does not pose an obstacle to its numerical solution. Indeed, the feature dimensionality $f$ can be avoided in its entirety using the now classical kernel representation of polynomials put forward by \citet{mercer1909functions}. Regression formulations amendable to such a kernel reformulation are typically referred to a kernel learning methods. It is thanks to both the flexibility of the regression formulation \eqref{eq:abstract_regression} and its computational tractability that polynomial and even more general kernel methods have experienced a lot of interest in the learning community \citep{suykens1999least}. Polynomial regression using kernel learning has indeed been used with success in many applications such as character recognition, speech analysis, image analysis, clinical diagnostics, person identification, machine diagnostics, and industrial process supervision. Today, efficient and mature software implementations of these so called polynomial kernel regressors are widely available, see c.f.\ \citep{scholkopf2002learning, pelckmans2002ls}. Unfortunately in high-dimensional settings $(f\gg n)$, the previously discussed curse of dimensionality and overfitting phenomena do pose a formidable obstacle to the recovery of the correct nonlinear relationship between input and response data. That is, in settings where we have many more monomial input features $f$ than observations $n$, it becomes unlikely that we recover a statistically meaningful regressor by solving \eqref{eq:abstract_regression}. 

Here we will work to address the previous issue by providing a sparse counterpart to the polynomial regression problem \eqref{eq:abstract_regression}. Sparse regression has recently been identified in the works of \citet{hastie2015statistical, candes2006robust} as an excellent antidote to the malignant phenomena of both dimensionality and overfitting. Interestingly, very few mature machine learning methods seem to have been developed which can deal reliably with sparse nonlinear regressors in a high-dimensional settings despite the obvious relevance of this problem class. One notable exception of direct relevance here is the \texttt{SPORE} algorithm by \cite{huang2010predicting} which uses an approach based on $\ell_1$-regularization. We subsequently describe a hierarchical sparse regression problem which controls both the dependence and functional complexity of the regression polynomials considered.

\subsection{Hierarchical $(k, \ell)$-Sparsity} 

The popularity and effectiveness of sparse regression can from a practical perspective be explained by the following two observations. In the digital age obtaining and processing vast amounts of input data is increasingly less of a burden. Nevertheless, we expect only a small number $k$ of all $p$ recorded inputs to be meaningfully correlated to the response data $Y$. The trouble is that we can not tell the relevant features from the obfuscating bulk of data ahead of time. Sparsity hence firstly describes the limited functional dependence between input and response data. As only a finite amount of data can be recorded, one must avoid overfitting by limiting the complexity of the considered functional relationships. Of the potential $f$ monomials making up the polynomial regressor $g$, the considered relationships should only depend on a select few of them. We use sparsity to enforce simplicity among the considered relationships. As both these described sparsity considerations are quite different in nature, we believe it is paramount not to conflate them. 

We will say that the function $g \in \mc P_{k, \ell}$ is so called $(k, \ell)$-sparse if it is the sum of $\ell$ monomials in at most $k$ inputs. For instance the regressor $g(x) = x_1^2 + x_2 x_3$ would be $(3,2 )$ sparse as it depends on the three inputs $x_1$, $x_2$ and $x_3$, and is made up of two monomials $x_1^2$ and $x_2x_3$. The resulting problem of hierarchical sparse regression can be cast as the regression problem
\begin{equation}
  \label{eq:abstract_sparse_regression}
  \min_{g \in \mc P_{k, \ell}} ~{\frac{1}{2}} \sum_{t\in[n]} \norm{y_t - g(x_t)}^2 + \frac{1}{2\gamma}\norm{g}^2.
\end{equation}
The previous regression formulation is a structured sparsity constrained version of \eqref{eq:abstract_regression}. Using this novel notion of hierarchical sparse regressors, we hope to keep the statistical power of parametric regression while simultaneously allowing highly nonlinear relationships between input and response data as well. Although structured hierarchical sparsity patterns were studied already by \cite{zhao2009composite}, they were never considered in our polynomial regression context directly. A related hierarchical kernel learning approach to a convex proxy of problem \eqref{eq:abstract_sparse_regression} is studied in \cite{bach2009exploring}. The regression problem \eqref{eq:abstract_sparse_regression} carries the additional benefit of automatically yielding highly interpretable regressors with only a few nonlinear terms and input dependencies. By explicitly controlling both the dependence complexity $k$ of used inputs as well as the functional complexity $\ell$ of the regression polynomials, the hierarchical sparse regression problem \eqref{eq:abstract_sparse_regression} promises to deliver nonlinear regressors with significant statistical power even in high-dimensional settings.

Unfortunately, solving the hierarchical sparse regression problem \eqref{eq:abstract_sparse_regression} can be challenging. Bringing to bear the power of modern integer optimization algorithms combined with smart heuristics, we will nevertheless show that many hierarchical sparse regression problems can nevertheless be handled.

\subsection{Exact Scalable Algorithms}

The problem of sparse linear regression has been studied extensively in the literature. Despite being provably hard in the sense of complexity theory's NP hardness, in practice many successful algorithms are available. Historically, the first heuristic methods for sparse approximation seem to have arisen in the signal processing community (c.f.\ the work of \citet{mallat1993matching} and references therein) and typically are of an iterative thresholding type. More recently, one popular class of sparse regression heuristics solve the convex surrogate
\begin{equation}
  \label{eq:abstract_l1_regression}
  \min_{\norm{g}_1 \leq \ell} ~{\frac{1}{2}} \sum_{t\in[n]} \norm{y_t - g(x_t)}^2 + \frac{1}{2\gamma}\norm{g}^2.
\end{equation}
to the sparse regression formulation \eqref{eq:abstract_sparse_regression}. Here the norm $\norm{g}_1$ of the polynomial $g$ is meant to denote the sum of the absolute values of its coefficients. The convex proxy reformulation \eqref{eq:abstract_l1_regression} is a direct adaptation of the seminal \texttt{Lasso} method of \citet{hastie2015statistical} to the polynomial regression problem \eqref{eq:abstract_regression}. The discussed \texttt{SPORE} algorithm by \cite{huang2010predicting} provides an implementation of this idea specific to the polynomial regression context considered here. Where the convex heuristic \eqref{eq:abstract_l1_regression} does not incorporate the hierarchical sparsity structure of our exact formulation \eqref{eq:abstract_sparse_regression}, more refined methods such as \texttt{Group Lasso} \citep{zhao2009composite, bach2008consistency} could in principle do so by considering structured norm constraints. There is an elegant theory for convex proxy schemes promising large improvements over the more myopic iterative thresholding methods. Indeed, a truly impressive amount of high-quality work \citep{buhlmann2011statistics, hastie2015statistical, wainwright2009sharp} has been written on characterizing when exact solutions can be recovered, albeit through making strong probabilistic assumptions on the data. 

The problem of exact sparse nonlinear regression however seems, despite its importance, not to have been studied extensively. Although they are well studied separately, combining nonlinear with sparse regression never received much attention. Our recent work \citep{bertsimas2016sparse} (and earlier in \citep{bertsimas2014statistics}) has revealed that despite complexity results, exact sparse linear regression is not outside the realm of the possible even for very high-dimensional problems with a number of features $f$ and samples $n$ in the 100,000s. Contrary to traditional complexity theory which suggests that the difficulty of a problem increases with size, the sparse regression problems seem to have the property that for a small number of samples $n$, exact regression is not easy to accomplish, but most importantly its solution does not recover the truth. However, for a large number of samples $n$, exact sparse regression can be done extremely fast and perfectly separates the true features in the data from the obfuscating bulk. These  results warrant also the possibility of nonlinear feature discovery in regression tasks.

\subsection{Triage Heuristic}

Unfortunately the dimension of hierarchical $(k, \ell)$-sparse regression problems quickly becomes problematic for all but midsize problem instances. The effective number of regressors $f$ is indeed combinatorial and hence scales quite badly in both the number of regressors $p$ as well of the degree $r$ of the considered polynomials. In order to provide a scalable algorithm to the problem of hierarchical kernel regression it is clear that the dimensionality of the problem needs to be reduced. 

\begin{figure}
  \centering
  \begin{center}
    \includegraphics[width=0.85\textwidth]{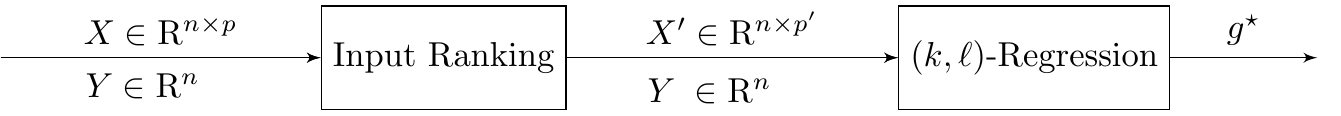}
  \end{center}
  \caption{Our two-step approach to $(k, \ell)$-sparse regression for observations $Y$ in $\Re^n$ and regressors $X$ in $\Re^{n\times p}$ . In a first step we select the $p'\ll p$ most promising inputs out of the $p$ candidates. The second step then performs exact $(k, \ell)$-sparse regression on these most promising candidates. We set out to show that this combination of a smart ranking heuristics and exact sparse regression goes a long way to solve hierarchical sparse regression problems \eqref{eq:abstract_sparse_regression} of practical size.}
  \label{fig:2-step-procedure}
\end{figure}

Our key insight in this paper will be to use polynomial kernel regression to rank the potential inputs. This heuristic method is very helpful in rejecting many irrelevant candidate inputs without missing out on the actual underlying nonlinear relationship between input and response data. Exact sparse hierarchical kernel regression described before will then be used to identify the relevant nonlinearities from among the most promising candidate inputs; see Figure \ref{fig:2-step-procedure}. In this paper we set out to show that a combination of smart heuristics and exact sparse regression goes a long way to solve hierarchical sparse regression problems \eqref{eq:abstract_sparse_regression} of practical size. 

\subsection{Contributions}

In this paper, we first and foremost want to promote a novel notion of hierarchical $(k, \ell)$-sparsity that rhymes well with the challenges of the big data era. Capturing limited functional dependence and complexity in big data problems is crucial to allow statistically meaningful regression in high-dimensional and nonlinear settings. Hierarchical $(k, \ell)$-sparse regression definition is in this regard a first step in the direction of lifting nonlinear regression into high-dimensional settings as well. In particular, we hope that the method presented here will show a more disciplined approach to nonlinear discovery than some more black box methods such as artificial neural networks.

Secondly, we also offer scalable algorithms able to solve these hierarchical $(k, \ell)$-sparse regression problems using modern optimization techniques. In accordance with previous results \citep{bertsimas2016sparse}, we show that exact sparse regression is not outside the realm of the possible even for very high-dimensional problems with a number of features $f$ and samples $n$ in the 100,000s. We will indeed show that we can reliably discover nonlinear relationships using a combination of smart heuristics and exact sparse regression using a cutting plane approach for convex integer optimization.

In order to judge the quality of a proposed regressor $g^\star$, we must measure on the one hand to what extent all the relevant monomial features are discovered. In order to do so, while at the same time avoiding notational clutter, we need to establish a way to refer to each of the monomials of degree $r$ in $p$ variables in an efficient fashion. Let $m_j:\Re^p\to\Re$ for each $j$ in $[f]$ denote a distinct monomial in $p$ inputs of degree at most $r$. We define the accuracy of a regressor $g^\star$ as
\begin{equation}
  A\% \defn \frac{\abs{\supp(g^\star) \cap \supp(g_{\mathrm true})}}{\abs{\supp(g_{\mathrm true})}},
\end{equation}
where $\supp(g)$ represents all $j$ such that monomial $m_j$ contributes to the polynomial $g$. This accuracy measure $A\%$ thus represents the proportion of true underlying monomial features discovered by the proposed polynomial regressor $g^\star$. On the other hand, we can use 
\begin{equation}
  F\% \defn \frac{\abs{\supp(g^\star) \setminus \supp(g_{\mathrm true})}}{\abs{\supp(g^\star)}}
\end{equation}
to quantify how many irrelevant features were wrongly included in the process. Perfect recovery occurs when the method gives the whole truth ($A\%=100$) and nothing but the truth $(F\%=0)$. In practice, however, machine learning methods must inevitable make a choice between both desirables.

We intend to illustrate that exact sparse regression methods have an inherent edge over proxy based sparse heuristics. Proxy based methods such as \texttt{Lasso} do indeed have several well documented shortcomings. First and foremost, as argued by \citet{bertsimas2014statistics} they do not recover very well the sparsity pattern. Furthermore, the \texttt{Lasso} leads to biased regression regressors, since the $\ell_1$-norm penalizes both large and small coefficients uniformly. The ability of our method to identify all relevant features is shown empirically to experience a phase transition. There exists a critical number of data samples $n_0$ such that when presented sufficient data $n>n_0$ our method recovers the ground truth ($A\%\approx 100$) completely, whereas otherwise its accuracy $A\%$ tends to zero. Crucially, the same number of samples $n_0$ also enables our method to reject most irrelevant features ($F\%\approx 0$) as well. We thus show that we significantly outperform \texttt{Lasso} in terms of offering regressors with a larger number of relevant features (bigger $A\%$) for far fewer nonzero coefficients (smaller $F\%$) enjoying at the same time a marginally better prediction performance. In the regime $n>n_0$ where our method is statistically powerful $(A\%\approx 100$, $F\%\approx 0$), its computational complexity is furthermore on par with sparse heuristics such as \texttt{Lasso}. This last observation takes away the main propelling justification for most heuristic based sparsity approaches.

\subsection*{Notation}

The knapsack set $\S^p_k$ denotes here the binary set $\S_k^p\defn\{s \in \{0,1\}^p:\textstyle\sum_{j\in[p]} s_j\leq k\}$, which contains all binary vectors $s$ selecting $k$ components out of $p$ possibilities. Assume that $(y_1, \dots, y_p)$ is a collection of elements and suppose that $s$ is an element of $\S^p_k$, then $y_s \in \Re^{\abs{s}}$ denotes the sub-collection of $y_j$ where $s_j = 1$. Similarly, we use $\supp(x) = \set{s \in \{0, 1\}^p}{s_j = 1 \iff x_j \neq 0} $ to denote those indices of a vector $x$ which are nonzero. We denote by $\S_+^n$ ($\S_{++}^n$) the cone of $n\times n$ positive semidefinite  (definite) matrices. Given a matrix $K\in\Re^{n\times n}$, we denote its element-wise $r$th power or Hadamard power as $K^{\circ r}$, i.e., we have that
\[
  K^{\circ r} \defn
\begin{pmatrix}
  K_{11}^r & K_{12}^r & \dots & K_{1n}^r \\
  K_{21}^r & K_{22}^r & \dots & K_{2n}^r \\
  \vdots & \vdots & \ddots & \vdots \\
  K_{n1}^r & K_{n2}^r & \dots & K_{nn}^r
\end{pmatrix}.
\]

\section{Hierarchical $(k, \ell)$-Sparse Polynomial Regression}
\label{sec:sparse_polynomial_regression}

In this section, we discuss two formulations of the hierarchical sparse regression problem \eqref{eq:abstract_sparse_regression} through a standard integer optimization lens. In the last twenty plus years the computational power of integer optimization solves has increased at a dramatic speed \citep{bertsimas2014statistics}. Where twenty years ago integer optimization for statistics was branded impossible, recent work \citep{bertsimas2016sparse, bertsimas2014statistics} has shown convincingly that this position needs to be revisited. The position that exact sparse regression is an unattainable goal only to be striven for via admittedly elegant convex heuristics such as \texttt{Lasso} should not be held any longer.

We considered in the sparse regression problem \eqref{eq:abstract_sparse_regression} as features the set of all monomials of degree at most $r$. It is clear that this sparse hierarchical regression problem over polynomials can equivalently stated as an optimization problem over their  coefficients in the monomial basis $m_j$ for $j$ in $[f]$. We will avoid the need to explicitely order each these monomials as follows. Define the dependents $D(i)$ of any data input $i$ as the set of indices $j$ such that the monomial $m_j$ depends on input $i$. Similarly, we define the ancestors $A(j)$ of any $j$ as the multiset of inputs making up the monomial $m_j$. Instead of using the classical monomial basis, we consider a scaled variant in which we take $m_j(\one_p)$ to coincide with the square root of the number of distinct ways to order the multiset $A(j)$. This rescaling of the monomials comes in very handy when discussing the solution of the regression problem \eqref{eq:abstract_regression} in Section \ref{sec:ranking}. In fact, the same scaling is implicitly made by the Mercer kernel based approach as well.

To make the discussion more concrete, we consider a data problem with $p=3$ inputs and associated data $(y_t, x_{t, 1}, x_{2, t}, x_{3, t})$ for $t \in [n]$. We consider all monomials on the three inputs of degree at most two, i.e., we consider the monomials and their corresponding indices as given below.

\begin{center}
  \begin{tabular}{|l|cccccccccc|}
    \hline
    Monomial $m_j$ & $1$ & $\sqrt{2}x_1$ & $\sqrt{2}x_2$ & $\sqrt{2}x_3$ & $x_1^2$ & $\sqrt{2}x_1 x_2$ & $\sqrt{2}x_1 x_3$ & $x_2^2$ & $\sqrt{2}x_2 x_3$ & $x_3^3$ \\
    Index $j$ & $1$ & $2$ & $3$ & $4$ & $5$ & $6$ & $7$ & $8$ & $9$ & $10$ \\
    \hline
  \end{tabular}
\end{center}
The set $D(i)$ corresponds to the set of indices of the monomials in which input $x_i$ participates. In our example above we have $D(1) = \{2, 5, 6, 7\}$ corresponding to the monomials $\{\sqrt{2} x_1, x_1^2, \sqrt{2} x_1 x_2, \sqrt{2} x_1 x_3 \}$. The set $A(j)$ corresponds to all inputs that are involved in the monomial with index $j$. Again as an illustration, in our example $A(7) = \{1, 3\}$ corresponding to the inputs $x_1$ and $x_3$. 

\subsection{Mixed Integer Formulation}

With the help of the previous definitions, we can now cast the hierarchical $(k, \ell)$-sparse regression problem \eqref{eq:abstract_sparse_regression} as a standard optimization problem. The problem of hierarchical sparse regression can indeed be cast as the following \ac{mio} problem
\begin{equation}
\label{eq:bigm:primal}
\begin{array}{rl}
	\min			& \frac{1}{2} {\displaystyle\sum_{t\in[n]}}  \Vert y_t - \sum_{j\in[f]} w_j \cdot m_j(x_t) \Vert^2 + \frac{1}{2\gamma} \norm{w}^2 \\[1em]
	\st			& w \in \Re^f, ~s \in \S^f_\ell, ~h\in \S^p_k, \\[0.5em]
				&  s_j \leq h_i\quad \forall i \in A(j), \hspace{1.05em} \forall j \in [f], \\[0.5em]
  				& -\mc M s_j \leq w_j \leq \mc M s_j \quad \forall j \in [f],
\end{array}
\end{equation}
using a big-$\mc M$ formulation. Its optimal solution $w^\star$ gives the coefficients of the polynomial $g^\star(x)=\sum_{j\in [f]} w^\star_j m_j(x)$ in \eqref{eq:abstract_sparse_regression} best describing the relationship between input and observations in our monomial basis $\{m_j\}$. The coefficient $w_j$ of any monomial $m_j$ is only then nonzero when $s_j=1$ as per the first constraint in \eqref{eq:bigm:primal} for a sufficiently large constant $\mc M$. This constant $\mc M$ needs to be estimated from data. Although nontrivial, this can be done using the results found in \citep{bertsimas2014statistics}. In any case, the binary variable $s \in \S^f_\ell$ represents the sparsity pattern in the monomials, i.e., which $\ell$ monomials are used out of $f$ potential candidates.  The ultimate constraint of formulation \eqref{eq:bigm:primal} encodes the hierarchical nature of our $(k, \ell)$-sparsity requirement. Only those monomials $m_j$ such that the input $h_i =1$ is selected for all its ancestors $i \in A(j)$ are considered as potential regressors. In all, the binary constraint $$(s, h) \in \S^{f, p}_{\ell, k} \defn \{s\in \S^f_\ell, h\in \S^p_k ~: ~s_j \leq h_i, \, \forall i \in A(j), \,\forall j \in [f]\}$$ hence represents the postulated hierarchical $(k, \ell)$-sparsity pattern.

To use the example discussed before in which we have three inputs and monomials of order two, the monomial $m_7(x) = \sqrt{2} x_1 x_3$ can only be included as a regressor if the variable $s_7 = 1$. The variable $s_7$ can only then be nonzero if both inputs $x_1$ and $x_2$ are selected which requires that the variables $h_1=1$ and $h_3=1$. The resulting optimal regressor polynomial $h^\star(x) = \sum_{j\in[f]} w_j^\star\cdot m_j(x)$ thus counts at most $\ell$ monomials depending on at most $k$ regressor inputs.

Although the direct formulation \eqref{eq:bigm:primal}  of the hierarchical $(k, \ell)$-sparse regression problem results in a well posed \ac{mio} problem, the constant $\mc M$ needs to be chosen with extreme care as not to impede its numerical solution. The choice of this data dependent constant $\mc M$ indeed affects the strength of the \ac{mio} formulation \eqref{eq:bigm:primal} and is critical for obtaining solutions quickly in practice \citep{bertsimas2014statistics}. Furthermore, as the regression dimension $p$ grows, explicitly constructing the \ac{mio} problem \eqref{eq:bigm:primal}, let alone solving it, becomes burdensome. In order to develop an exact scalable method a different perspective on sparse regression is needed. In the subsequent sections we develop an exact method which avoids using a big-$\mc M$ formulation while at the same time avoiding explicit construction of the problem at solve time.

\subsection{Convex Integer Formulation}

We next establish that the sparse regression problem \eqref{eq:bigm:primal} can in fact be represented as a pure binary optimization problem. By doing so we will eliminate the dependence of our formulation on the data dependent constant $\mc M$. Our results here closely resemble those presented in \citet{bertsimas2016sparse}.
We will need the help of the following supporting lemma regarding linear regression.

\begin{lemma}[The regression loss function $c$]
\label{lemm:convexity}
The least-squares regression cost $c(Z Z\tpose) \defn \min_w \frac{1}{2} \norm{Y-Z w}^2 + \frac{1}{2\gamma} \norm{w}^2$ admits the following explicit characterization
\begin{align}
		c(Z Z\tpose) &= \frac12 Y\tpose \left(\eye{n} + \gamma Z Z\tpose \right)^{-1} Y. \label{eq:char2}
\end{align}
\end{lemma}
\begin{proof}
As the regression problem over $w$ in $\Re^p$ is an unconstrained \ac{qo} problem, the optimal value $w^\star$ satisfies the linear relationship $(\eye{p}/\gamma + Z\tpose Z)w^\star = Z\tpose Y.$ Substituting the expression for the optimal linear regressor $w^\star$ back into optimization problem, we arrive at 
\[
	c(Z Z\tpose) = 1/2 Y\tpose Y - 1/2 Y\tpose Z \left(\eye{p}/\gamma + Z\tpose Z\right)^{-1} Z\tpose Y.
\]
The final characterization can be derived from the previous result with the help of the matrix inversion lemma found stating the identity 
\(
\left(\eye{n} + \gamma Z Z\tpose\right)^{-1} = \eye{n} - Z \left( \eye{p}/\gamma + Z\tpose Z\right)^{-1} Z\tpose.
\)
\qed        
\end{proof}

Lemma \ref{lemm:convexity} will enable us to eliminate the continuous variable $w$ out of the \ac{mio} sparse regression formulation \eqref{eq:bigm:primal}. The following result provides a different pure integer approach to hierarchical sparse regression. It will form the basis to our attempts to solve hierarchical regression problems.

\begin{theorem}[Hierarchical $(k, \ell)$-sparse regression]
\label{thm:sparse_regression}
The hierarchical $(k, \ell)$-sparse regression problem \eqref{eq:abstract_sparse_regression} can be reformulated as the pure \ac{cio} problem 
\begin{equation}
\label{eq:opt:miop:kernel}
\begin{array}{rl}
	\min 	& \displaystyle \frac12 Y\tpose \left(\eye{n} + \gamma \textstyle\sum_{j\in[f]} s_j K_j \right)^{-1} Y \\[0.5em]
	\st    	& s\in \S^f_\ell, ~ h \in \S^p_k, \\[0.5em]
		& s_j \leq h_i\quad \forall i \in A(j), \hspace{1.05em} \forall j \in [f],
\end{array}
\end{equation}
where the micro kernel matrices $K_j$ in $\S^n_+$ are defined as the dyadic outer products $K_j \defn m_j (X) \cdot m_j(X)\tpose$. 
\end{theorem}
\begin{proof}
We start the proof by separating the optimization variable $w$ in the sparse regression problem \eqref{eq:bigm:primal} into its support $s \defn \supp{w}$ and the corresponding non-negative entries $w_s$. Evidently, we can now write the sparse regression problem \eqref{eq:bigm:primal} as the bilevel minimization problem 
\begin{equation}
\label{eq:reformulation}
\min_{s, h}\left[\min_{w \in \Re^k} ~ \frac{1}{2\gamma} \norm{w}^2 + \frac{1}{2} \sum_{t\in[n]}  \Vert y_t - \textstyle\sum_{\set{j \in [n]}{s_j=1}} w_{j} \cdot m_j(x_t) \Vert^2 \right].
\end{equation}
It now remains to be shown that the inner minimum can be found explicitly as the objective function of the optimization problem \eqref{eq:opt:miop:kernel}. Using Lemma \ref{lemm:convexity}, the minimization problem can be reduced to the minimization problem $\min \{c(m_s(X)\cdot m_s(X)\tpose) : (s, h) \in \mc \S^{p, r}_{k, \ell} \}$. We finally remark that the outer product can be decomposed as the sum $m_s(X)\cdot m_s(X)\tpose = \textstyle\sum_{j\in[p]} s_j \cdot m_j(X) \cdot m_j(X)\tpose$, thereby completing the proof. \qed
\end{proof}

\citet{bertsimas2016sparse} provide an algorithm which can solve a related sparse linear regression problem up to dimensions $f$ and $n$ in the order of $100,000$s based on a cutting plane formulation for integer optimization. Contrary to traditional complexity theory which suggests that the difficulty of a problem increases with size, there exists a critical number of observations $n_0$ such that the sparse regression problems seem to have the property that for a small number of samples $n < n_0$, an exact regressor is not easy to obtain, but most importantly its solution does not recover the truth ($A\%\approx 0$ and $F\%\approx 100$). For a large number of samples $n > n_0$ however, exact sparse regression can be done extremely fast and perfectly separates ($A\%\approx 100$ and $F\%\approx 0$) the true monomial features from the obfuscating bulk. These results warrant the possibility of nonlinear feature discovery for regression task of practical size as well.

Despite the previous encouraging results, for all but midsize problems hierarchical sparse regression quickly becomes problematic. The effective number of regression features $f$ is indeed combinatorial in the number of inputs $p$ and degree $r$ of the considered polynomials. Our key insight is to triage the inputs first heuristically using an efficient input ranking method described in the subsequent section. Later in Section \ref{sec:empirical_results} we will show that this two-step procedure outlined in Figure \ref{fig:2-step-procedure} goes a long way to solve practical hierarchical sparse regression problems.

\section{Polynomial Kernel Input Ranking}
\label{sec:ranking}

The objective in this section is to present an efficient method which can address the high-dimensional nature of exact sparse regression by ignoring irrelevant regression inputs and working with promising candidates only. The reader might wonder at this point whether any such attempt would not entail the solution of the original hierarchical sparse regression problem. Here, however, we do not claim to triage the inputs optimally, but rather aim for a simple approximate yet fast method. We will attempt to do so by leveraging the fact that the nonlinear regression problem \eqref{eq:abstract_regression} without sparse constraints can be solved efficiently.

A seminal result due to \citet{vapnik1998support} states that the feature dimensionality $f$ of the unconstrained regression problem \eqref{eq:abstract_regression} surprisingly does not play any role in its numerical solution. Indeed, the feature dimensionality $f$ can be done away with in its entirety using the now classical \citet{mercer1909functions} kernel representation theorem. We can state the polynomial regression problem \eqref{eq:abstract_regression} as an optimization problem in terms of  coefficients in the monomial basis
\begin{equation}
\label{eq:linear_regression:primal}
\begin{array}{rl}
	\min		& \frac{1}{2} \sum_{t\in[n]} \Vert y_t - \sum_{j\in[f]} w_j \cdot m_j(x_t) \Vert^2 + \frac{1}{2\gamma} \norm{w}^2 \\[0.5em]
	\st			& w \in \Re^f. \\
\end{array}
\end{equation}
We state the Mercer kernel representation in Theorem \ref{thm:vapnik} for the sake of completeness regarding the dual of the regression problem \eqref{eq:linear_regression:primal}.  It should be noted that surprisingly the dimension $f$ does not play a role but instead the number of samples $n$ is of importance. This previous observation is what has propelled kernel learning algorithms as viable nonlinear regression methods \citep{scholkopf2002learning}.

\begin{theorem}[Mercer Kernel Representation \citep{vapnik1998support}]
\label{thm:vapnik}
The polynomial regression problem \eqref{eq:linear_regression:primal} can equivalently be formulated as the unconstrained maximization problem 
\begin{equation}
\label{eq:regression:dual}
\begin{array}{rl}
	c(K) = \max	& -\frac{\gamma}{2} \alpha\tpose K \alpha  - \frac{1}{2} \alpha\tpose \alpha + Y\tpose \alpha  \\[0.5em]
	\st    	& \alpha \in \Re^n, \\
\end{array}
\end{equation}
where the positive semidefinite kernel matrix $K \defn m(X)\cdot m(X)\tpose$ allows for an efficient characterization as the Hadamard power
\(
	K = (X X\tpose + \mb 1_{n\times n})^{\circ r}.
\)
\end{theorem}

Theorem \ref{thm:vapnik} uses the Mercer kernel representation which establishes that the outer product $m(X)\cdot m(X)\tpose$ can be characterized as the element-wise Hadamard power $(X X\tpose + \mb 1_{n\times n})^{\circ r}$ for our specific polynomial bases. Indeed, for any $t$ and $t'$ in $[n]$ we have
\begin{align*}
  K(t, t') \defn & ~m(x_t)\tpose \cdot m(x_{t'}) \\
  = & ~[\sqrt{\norm{ A(1) }}\cdot m_1(x_t), \dots, \sqrt{\norm{A(f)}} \cdot m_f(x_t)]\tpose \cdot [\sqrt{\norm{A(1)}}\cdot m_1(x_{t'}), \dots, \sqrt{\norm{A(f)}}\cdot m_f(x_{t'})] \\
  = & ~\textstyle\sum_{j\in[f]} \norm{A(j)} m_j(x_t)\cdot m_j(x_{t'}) \\
  = & ~\textstyle\sum_{j\in[f]} \norm{A(j)} m_j([x_{t, 1}\cdot x_{t', 1}, \dots, x_{t, p}\cdot x_{t', p}]) \\
  = & ~(1 + \textstyle \sum_{i\in[p]} x_{t, i}\cdot x_{t', i})^r\\
  = & ~(1+x_t\tpose x_{t'} )^r
\end{align*}
where $\norm{A(j)}$ denotes here the number of distinct ways to order the multiset $A(j)$. The penultimate equality is recognized as the binomial expansion theorem. Note that for the Mercer kernel representation to hold, the monomial basis had indeed to be properly normalized using $\sqrt{\norm{A(j)}}$ as explained in the beginning of Section \ref{sec:sparse_polynomial_regression}. This well known but crucial observation seems to have been made first by \citet{poggio1975optimal}.

The size of the kernelized regression problem \eqref{eq:regression:dual} scales only with the number of data points $n$ rather than the feature dimension $f$. It could be remarked that as the kernelized regression problem is unconstrained it admits a closed form solution in the form of the linear system 
\(
	(\eye{n} + \gamma K) \alpha^\star = Y.
\)
The optimal regression coefficients $w^\star$ in formulation \eqref{eq:linear_regression:primal} are linearly related to the optimal dual variable $\alpha^\star$ in formulation \eqref{eq:regression:dual} via the complementarity conditions which here read
\begin{equation}
\label{eq:complementarity}
	w_j^\star = \gamma \cdot m_j(X)\tpose \alpha^\star.
\end{equation}
Although this last relationship is linear, computing the coefficients in the monomial basis might still prove a daunting task merely because of the shear number of them. In the following, we show that we can nevertheless compute the Euclidean norm of the optimal coefficients $w^\star_j$ in front of the monomials depending on a certain input $i$ efficiently. 

\begin{proposition}
\label{eq:coefficients}
The Euclidean norm of the coefficients $w^\star_j$ in front of all monomials $m_j$ depending on input $i$ is related to the dual optimal variable $\alpha^\star$ in \eqref{eq:regression:dual} as
\begin{equation}
\label{eq:coefficients_size}
	\textstyle \norm{w^\star_{D(i)}}^2 = \sum_{j \in D(i)} \norm{w_j^\star}^2 = \gamma^2 \cdot \alpha^{\star\top} K_i \,\alpha^\star,
\end{equation}
where the kernel matrix $K_i$ can be characterized explicitly as $K -  (X X\tpose - X_i X_i\tpose + \mb 1_{n\times n})^{\circ r}$.
\end{proposition}
\begin{proof}
From the linear relationship \eqref{eq:complementarity} between the optimal coefficients $w^\star$ and dual variable $\alpha$ it follows immediately that $\norm{w_j}^2 = \gamma^2 \cdot \alpha^{\star\top} (\sum_{j\in D(i)}K_j) \,\alpha^\star$. Through simple expansion it is quite easy to see that $(X X\tpose - X_i X_i\tpose + \mb 1_{n\times n})^{\circ r}$ coincides exactly with $\sum_{j \notin D(i)} K_j$. Hence, the kernel matrix $K_i \defn \sum_{j\in D(i)}K_j$ is found as its complement $K - (X X\tpose - X_i X_i\tpose + \mb 1_{n\times n})^{\circ r}$. \qed
\end{proof}

Hence, despite the fact that the size of the optimal coefficients $w_j$ in front of the monomials depending on a certain input $i$ consists of the sum of squares of as many as $\binom{p+r-1}{r-1}$ components, it can nevertheless be computed without much effort. Unfortunately, the optimal regressors coefficients $w^\star$ in \eqref{eq:linear_regression:primal} are not expected to be sparse. Nevertheless, the optimal regressors coefficients can be used to provide a ranking of the importance of the $p$ data inputs. The Euclidean norm of the coefficients of the monomials which depend on input $i$ can indeed be used as a proxy for the relevance of the input of interest.  Fortunately, the quantities \eqref{eq:coefficients_size} are very efficient to compute once the optimal dual variable $\alpha^\star$ has been found. Indeed, each quantity can be computed in $\mc O(n^2)$ time independent of the dimension $f$ of the polynomials considered. 

The complete computation is given in Algorithm \ref{alg:variable_ranking}. Though not exact, it gives a good indication of the significance of each of the $p$ inputs. In fact it is very closely related the backward elimination wrapper methods discussed in \citep{guyon2003introduction}.

\begin{algorithm}
\SetKwInOut{Input}{input}
\SetKwInOut{Output}{output}
\caption{Input Ranking}
\label{alg:variable_ranking}
 \Input{$Y \in \Re^{n}$, $X \in \Re^{n\times p}$ and $r \in \N$}
 \Output{$r \in \Re^p$}
 $K \gets (X X\tpose + \mb 1_{n\times n})$ \\
 $\alpha^\star = (\eye{n} + \gamma K^{\circ r})^{-1} Y$ \\
 \For {$i$ in $[p]$ } {
 	$K_i \gets K^{\circ r} - (K-X_i X_i\tpose)^{\circ r}$ \\[0.1em]
 	$r_i \gets \gamma^2 \cdot  \alpha^{\star\top} K_i \alpha^\star$ \\
 }
\end{algorithm}

An alternative way of looking at our input ranking algorithm is through the subgradients of the convex regression loss function $c$ defined in the dual problem \eqref{eq:regression:dual}. We note that the subgradient of the function $c$ can be computed explicitly using its dual characterization given in \eqref{eq:regression:dual} as well.

\begin{proposition}[Derivatives of the optimal regression loss function]
\label{lemm:derivatives}
We have that the subgradient of the regression loss function $c$ as a function of the kernel $K$ can be stated as
\[
	\nabla c =  -\frac{\gamma}{2} \cdot \alpha^{\star \top} \nabla K \alpha^\star,
\]
where $\alpha^\star$ maximizes \eqref{eq:regression:dual}.
\end{proposition}
\begin{proof}
From the dual definition of the regression loss function $c$ it is clear that we have the inequality  
\begin{equation}
\label{eq:subgradient}
	c(\bar K) \geq -\frac{\gamma}{2} \cdot \alpha^{\star \top} \bar{K} \alpha^\star - \frac12 \alpha^{\star\top} \alpha^\star + Y\tpose \alpha^\star
\end{equation}
for all $\bar K$. From the very definition of $\alpha^\star$ as the maximizer of \eqref{eq:regression:dual} it follows that the previous inequality becomes tight for $\bar K = K$. This proves that the left hand side of \eqref{eq:subgradient} is a subgradient to the regression loss function $c$ at the point $K$. \qed
\end{proof} 

When comparing the previous proposition with the result in Theorem \ref{eq:coefficients}, it should be noted that the derivatives of $c$ at $K$ agree up to the constant $-2\gamma$ with the sum of squares of the coefficients $w^\star$ of the optimal polynomial. In essence thus, our Algorithm \ref{alg:variable_ranking} ranks the inputs according to the linearized loss of regression performance $\nabla c_i$ caused by ignoring the inputs using the polynomial regression \eqref{eq:abstract_regression}. The quantity $r_i$ characterized up to first-order the loss in predictive power when not using the input $i$. The higher this caused loss, the more importance is assigned to including input $i$ as a regressor. 

As was noted in the beginning of the section, the input ranking method presented in Algorithm \ref{alg:variable_ranking} does not aspire to find all $k$ relevant inputs exactly. Rather, it is only meant to eliminate the most unpromising inputs and keep $p'$ high potential candidates as illustrated in Figure \ref{fig:2-step-procedure}. Among those inputs which are deemed promising we will then solve the hierarchical $(k, \ell)$ sparse regression problem \eqref{eq:abstract_sparse_regression} exactly as explained in the subsequent section.

\section{A Cutting Plane Algorithm for Hierarchical Sparse Regression}
\label{sec:cutting_plane}

We point out again that our pure integer formulation \eqref{eq:opt:miop:kernel} of the hierarchical $(k, \ell)$-sparse regression problem \eqref{eq:abstract_sparse_regression} circumvents the introduction of a big-$\mc M$ constant  which is simultaneously hard to estimate and crucial for its numerical efficacy. Nevertheless, explicitly constructing the optimization problem \eqref{eq:opt:miop:kernel} results in the integer semidefinite optimization problem $$\min_{(s, h)\in \S^{f, p}_{\ell, k}} c\left(\textstyle\sum_{j\in[f]} s_j K_j\right)$$ which might prove daunting. The regression loss function $c$ is indeed a semidefinite representable function \citep{nesterov1994interior} to be optimized over the discrete set $\S^{f, p}_{\ell, k}$. Without even taking into account the discrete nature of the optimization problem \eqref{eq:opt:miop:kernel}, solving a \ac{sdo} of size the number of samples might even prove in a convex case impractical for medium size problems with $n \approx 1,000$. In order to solve our \ac{cio} formulation \eqref{eq:opt:miop:kernel}, we take here an alternative route using the outer approximation approach introduced by \citet{duran1986outer}. 

The outer approximation algorithm proceeds to find a solution to the \ac{cio} problem \eqref{eq:opt:miop:kernel} by constructing a sequence of piecewise affine approximations $c^a$ to the loss function $c$ based on cutting planes. From the pseudocode given in Algorithm \ref{alg:outer_approximation}, the outer approximation Algorithm \ref{alg:outer_approximation} can be seen to construct an increasingly better piece-wise affine lower approximation to the convex regression loss function $c$ using the subgradients defined in Proposition \ref{lemm:derivatives}. At each iteration indeed, the cutting plane added $\eta \geq c(s) + \nabla c (s^a) (s-s^a)$ in the outer approximation Algorithm \ref{alg:outer_approximation} cuts off the current binary solution $s^a$ unless it happened to be optimal in \eqref{eq:opt:miop:kernel}. As the algorithm progresses, the outer approximation function $c^a$ thus constructed
\[
  c^a(\textstyle\sum_{j\in [f]} s_j K_j) \defn \max_{u\in[a]} \, c(\textstyle\sum_{j\in[f]}s^u_{j} K_j) + \nabla c(\textstyle\sum_{j\in[f]}s^u_j K_j) (\textstyle\sum_{j\in[f]}(s-s^u_j) K_j)
\]
becomes an increasingly better approximation to our regression loss function of interest. Unless the current binary solution $s^a$ is optimal in \eqref{eq:opt:miop:kernel}, a new distinct cutting plane will refine the approximation. The main advantage of working with the approximations $c^a$ instead of $c$ is that the former results in linear integer optimization instead of the much more tedious semidefinite integer optimization problem over $\S^{f, p}_{\ell, k}$.

\begin{theorem}[Exact Sparse Regression \citep{fletcher1994solving}]
  Algorithm \ref{alg:outer_approximation} returns the exact sparse solution $w^\star$ of the hierarchical $(k, \ell)$-sparse regression problem \eqref{eq:bigm:primal} in finite time.
\end{theorem}

\begin{algorithm}
\SetKwInOut{Input}{input}
\SetKwInOut{Output}{output}
\caption{The outer approximation process}
\label{alg:outer_approximation}
 \Input{$Y \in \Re^{n}$, $X \in \Re^{n\times p}$ and $k \in [1, p]$}
 \Output{$s^\star \in \S^p_k$ and $w^\star \in \Re^p$}
 $s_1 \gets$ warm start \\ $\eta_1 \gets 0$ \\ $a \gets 1$ \\
 \While{$\eta_a < c(s_a)$}{
   $ s_{a+1}, ~\eta_{a+1} \gets \arg \min_{s, \, \eta} \, \{ \, \eta \in \Re_+ ~\mathrm{s.t.} ~(s, h) \in \S^{f, p}_{\ell, k},  ~~\eta \geq c(s^u) + \nabla c(s^u) (s-s^u), ~ \forall u \in [a] \}$ \\
   $a \gets a+1$
  }
 $s^\star \gets s^{a}$ \\
 $w^\star \gets 0$, \quad $w^\star_{s^\star} \gets \left(\eye{p}/\gamma + X_{s^\star}\tpose X_{s^\star} \right)^{-1} X_{s^\star}\tpose Y$
\end{algorithm}

Despite the previous encouraging corollary, it nevertheless remains the case that from a theoretical point of view we may need to compute exponentially many cutting planes in the worst-case, thus potentially rendering our approach impractical. Indeed, in the worst-case Algorithm \ref{alg:outer_approximation} considers all integer point in $\S^{f, p}_{\ell, k}$ forcing us to minimize the function so constructed
\begin{equation*}
  \bar{c}(s) \defn \max_{(\bar s, \bar h)\in \mc \S^{f, p}_{\ell, k}} ~c(\bar s) + \nabla c (\bar s) (s-\bar s)
\end{equation*}
over the hierarchical binary constraint set $\S^{f, p}_{\ell, k}$. As the number of integer points in the constraint set $\S^{f, p}_{\ell, k}$ is potentially extremely large, the previous full explicit construction should evidently be avoided. In practice usually very few cutting planes need to be considered making the outer approximation method an efficient approach.

Furthermore, at each of the iterations in Algorithm \ref{alg:outer_approximation} we need to solve the convex integer optimization problem $\min \{ c^a(\textstyle\sum_{j\in[f]} s_j K_j) : (s, h)\in \S^{f, p}_{\ell, k} \}$. This can be done by constructing a branch-and-bound tree, c.f. \cite{lawler1966branch}, which itself requires a potential exponential number of its leaves to be explored. This complexity behavior is however to be expected as exact sparse regression is known to be an NP-hard problem. Surprisingly, the empirical timing results presented in Section \ref{sec:empirical_results} suggest that the situation is much more interesting than what complexity theory might suggest. In what remains of this section, we briefly discuss a technique to carry out the outer approximation algorithm more efficiently than a naive implementation would.

In general, outer approximation methods such as Algorithm \ref{alg:outer_approximation} are known as \textit{multi-tree} methods because every time a cutting plane is added, a slightly different integer optimization problem is to be solved anew by constructing a branch-and-bound tree. Consecutive integer optimization problems $\min \{ c^a(\textstyle\sum_{j\in[f]} s_j K_j) : (s, h)\in \S^{f, p}_{\ell, k} \}$ in Algorithm \ref{alg:outer_approximation} differ only in one additional cutting plane. Over the course of our iterative cutting plane algorithm, a naive implementation would require that multiple branch and bound trees are built in order to solve the successive integer optimization problems. We implement a \textit{single tree} way of solving the iteration algorithm \ref{alg:outer_approximation} by using dynamic constraint generation, known in the optimization literature as either a lazy constraint or column generation method. Lazy constraint formulations described in \cite{barnhart1998branch} dynamically add cutting planes to the model whenever a binary feasible solution is found. This saves the rework of rebuilding a new branch-and-bound tree every time a new binary  solution is found in Algorithm \ref{alg:outer_approximation}. Lazy constraint callbacks are a relatively new type of callback. To date, the only commercial solvers which provide lazy constraint callback functionality are \texttt{CPLEX}, \texttt{Gurobi} and \texttt{GLPK}.

\section{Numerical results}
\label{sec:empirical_results}

To evaluate the effectiveness of hierarchical sparse polynomial regression discussed in this paper, we report its performance first on synthetic sparse data and subsequently on real data from the \texttt{UCI Machine Learning Repository} as well. All algorithms in this document are implemented in \texttt{Julia} and executed on a standard \texttt{Intel(R) Xeon(R) CPU E5-2690 @ 2.90GHz} running \texttt{CentOS release 6.7}. All optimization was done with the help of the commercial mathematical optimization distribution \texttt{Gurobi version 6.5} interfaced through the \texttt{JuMP} package developed by \citet{LubinDunningIJOC}.

\subsection{Benchmarks and Data}
\label{ssec:algorithms}

In the first part we will describe the performance of our cutting plane algorithm for polynomial sparse regression on synthetic data. We first describe the properties of the synthetic data in more detail.

\paragraph{Synthetic data: }The synthetic observations $Y$ and input data $X$ satisfy the linear relationship
\begin{align*}
  Y & = g_{\mathrm{true}}(X) + E\\
   & = m(X)\cdot w_{\mathrm{true}} + E.
\end{align*}
The unobserved true regressor $w_{\mathrm{true}}$ has exactly $\ell$ nonzero components at indices $j$ selected uniformly at random without replacement from $\mc J$. The previous subset $\mc J$ is itself furthermore constructed randomly as $\cup_{i \in \mc I} D(i)$ where the $k$ elements in $\mc I$ are uniformly selected out of $[p]$. The previous discussed construction thus guarantees that the ground truth $w_{\mathrm{true}}$ is $(k, \ell)$ sparse. Additionally, the nonzero coefficients in $w_{\mathrm{true}}$ are drawn uniformly at random from the set $\{-1, +1\}$. The observation $Y$ consists of the signal $S \defn X w_{\mathrm{true}}$ corrupted by the noise vector $E$. The noise components $e_t$ for $t$ in $[n]$ are drawn \ac{iid} from a normal distribution and scaled such that the signal-to-noise ratio equals $$\sqrt{\mathrm{SNR}} \defn \norm{S}_2 / \norm{E}_2.$$ Evidently as the signal-to-noise ratio $\mathrm{SNR}$ increases, recovery of the unobserved true regressor $w_{\mathrm{true}}$ from the noisy observations can be done with higher precision.
We have yet to specify how the input matrix $X$ is chosen. We assume here that the input data samples $X = (x_1, \dots, x_n)$ are drawn from an \ac{iid} source with Gaussian distribution. Although the columns of the data matrix $X$ are left uncorrelated, the features $m_j(X)$ will be correlated. For instance, it is clear that the second and forth power of the first input can only be positively correlated.

We will compare the performance of hierarchical sparse regression with two other benchmark regression approaches. These two approaches where chosen as to investigate the impact of both sparsity and nonlinearity on the performance of our method. Next we describe the particularities of these two benchmarks more closely.

\paragraph{Polynomial Kernel Regression:} As a first benchmark we consider polynomial regression defined explicitly in \eqref{eq:abstract_regression}. The primary advantage of this formulation stems from the fact that the optimal polynomial regressor in \eqref{eq:abstract_regression} can be found efficiently using
\[
  g_2^\star(x) = \gamma \textstyle \sum_{t\in[n]} \alpha^\star_t (x\tpose x_t+1)^r
\]
where $\alpha^\star$ is the maximizer of \eqref{eq:regression:dual}. As this formulation does not yield sparse regressors, this benchmark will show us the merit of sparsity in terms of prediction performance. Classical Ridge regression is found as a special case for $r=1$, enabling us to see what fruits nonlinearity brings us.

\paragraph{$\ell_1$-Heuristic Regression:} In order to determine the effect of exact sparse regression in our two step procedure, we will also use a close variant of the \texttt{SPORE} algorithm developed by \citet{huang2010predicting} as a benchmark. Using the input ranking method discussed in Section \ref{sec:ranking}, we determine first the $p'$ most relevant inputs heuristically. Using the remaining inputs $X' \in \Re^{n\times p'}$ and response data $Y\in\Re^n$ we then consider the maximizer $g_1^\star$ of \eqref{eq:abstract_l1_regression} as a heuristic sparse regressor. This two-step regression procedure hence shares the structure outlined in Figure \ref{fig:2-step-procedure} with our hierarchical exact sparse regression algorithm. As to have a comparable number of hyper parameters as our two-step approach, we finally perform Ridge regression on the thus selected features using a Tikhonov regularization parameter $\gamma$ selected using cross validation.

Theoretical considerations \citep{buhlmann2011statistics, hastie2015statistical, wainwright2009sharp} and empirical evidence \citep{donoho2006breakdown}  suggests that the ability to recover the support of the correct regressor $w_{\mathrm true}$ from noisy data using the \texttt{Lasso} heuristic experiences a phase transition. While it is theoretically understood \citep{gamarnik2017high} that a similar phase transition must occur in case of exact sparse regression, due to a lack of scalable algorithms such a transition was never empirically reported. The scalable cutting plane algorithm developed in Section \ref{sec:cutting_plane} offers  us the means to do so however. Our main observation is that exact regression is significantly better than convex heuristics such as \texttt{Lasso} in discovering all true relevant features ($A\%\approx 100$), while truly outperforming their ability to reject the obfuscating ones ($F\%\approx 0$).

\subsection{Phase Transitions}

\begin{figure}
\begin{center}
    \includegraphics[width=0.85\textwidth]{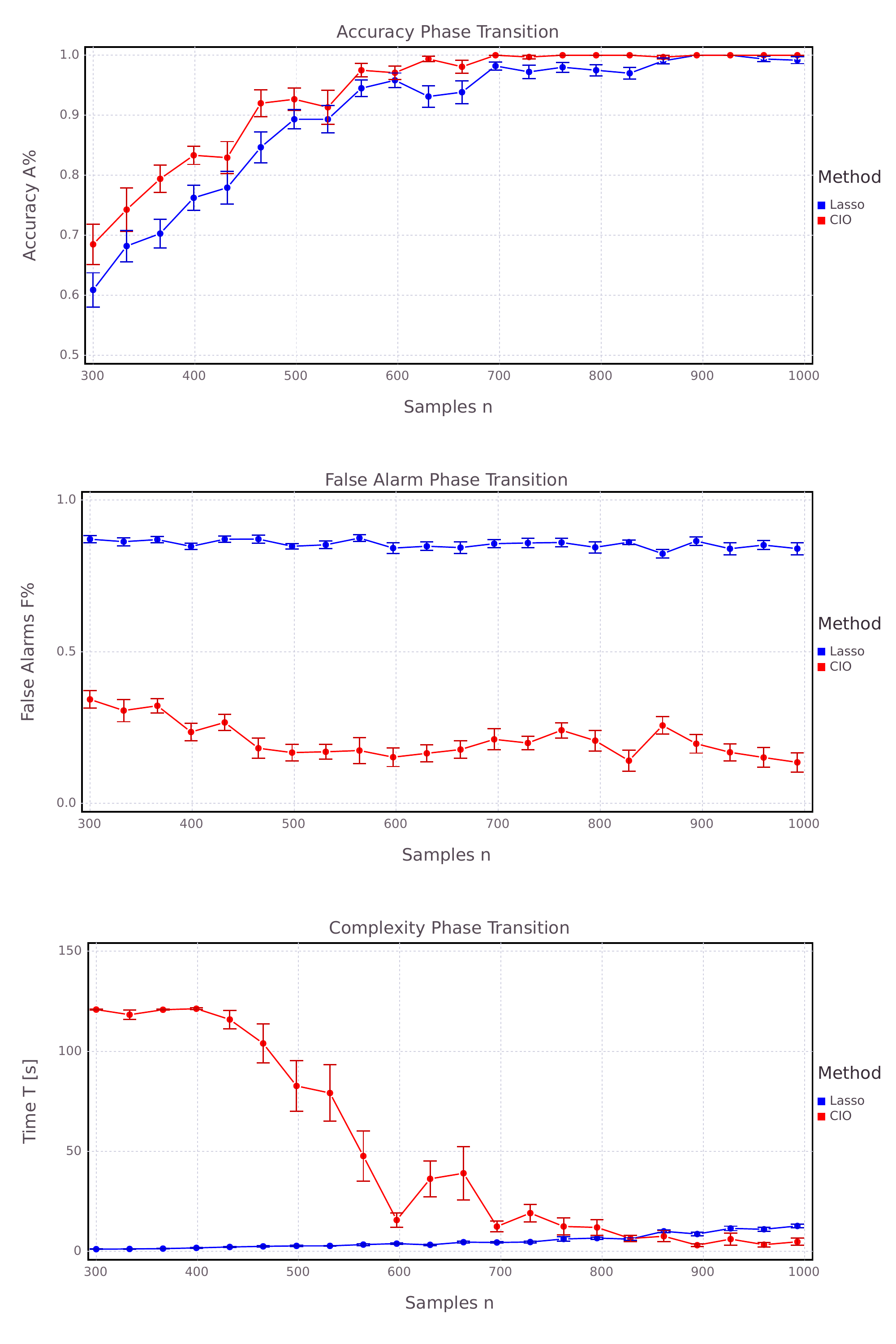}
\end{center}
\caption{The performance of exact sparse regression and the \texttt{Lasso} heuristic on synthetic data in terms of  accuracy $A\%$, false alarm rate $F\%$ and time $T$ in seconds.}
\label{fig:phase-transition}
\end{figure}

In Figure \ref{fig:phase-transition} we depict the performance of the hierarchical sparse regression procedure outlined in Figure \ref{fig:2-step-procedure} in its ability to discover all relevant regression features ($A\%$) and the running time $T$ in seconds as a function of the sample size $n$ for a regression problem with $p'=p=25$ inputs expanded with the help of all cubic monomials into $f=3276$ possible features. As $p'=p$ the input ranking heuristic is irrelevant here and instead all inputs are considered by the exact sparse regression procedure. The reported results are the average of $20$ independent sparse synthetic data sets where the error bars vizualize the inter data set variation. For the purpose of this section, we assume that we know that only $\ell=20$ features are relevant but do not know which. In practice though, the parameter $\ell$ must be estimated from data. In order to play into the ballpark of the \texttt{Lasso} method, no hierarchical structure $(k=p)$ is imposed. This synthetic data is furthermore lightly corrupted by Gaussian noise with $\sqrt{SNR}=20$. Furthermore, if the optimal sparse regressor was not found by the outer approximation Algorithm \ref{alg:outer_approximation} within two minutes, the best solution found up to that point is considered.

It is clear that our ability to uncover all relevant features ($A\%\approx 100)$ experiences a phase transition at around $n_0 \approx 600$ data samples. That is, when given more than $n_0$ data points, the accuracy of the sparse regression method is perfect. With fewer data points our ability to discover the relevant features quickly diminishes. For comparison, we also give the accuracy performance of the \texttt{Lasso} heuristic described in \eqref{eq:abstract_l1_regression}. It is clear that exact sparse regression dominates the \texttt{Lasso} heuristic and needs fewer samples for the same accuracy $A\%$.

Especially surprising is that the time $T$ it takes Algorithm \ref{alg:outer_approximation} to solve the corresponding $(k, \ell)$-sparse problems exactly experiences the same phase transition as well. That is, when given more than $n_0$ data points, the accuracy of the sparse regression method is not only perfect but easy to obtain. In fact, in case $n>n_0$ our method is empirically as fast as the \texttt{Lasso} based heuristic. This complexity transition can be characterized equivalently in terms of the number of cutting planes necessary for our outer approximation Algorithm \ref{alg:outer_approximation} to return the optimal hierarchical sparse regressor. While potentially exponentially many ($|{S^{f, p}_{\ell, k}}|$ in fact) cutting planes might be necessary in the worst-case, Table \ref{tab:cuts} list the actual average number of cutting planes considered on the twenty instances previously discussed in this section. When $n>n_0$ only a few cutting planes suffice, whereas for $n<n_0$ an exponential number seem to be necessary.

\begin{table}
  \centering
  \begin{tabular}{|lcccccccc|}
    \hline
    Samples $n$ & $300$ & $400$ & $500$ & $600$ & $700$ & $800$ & $900$ & $1000$ \\
    Cutting planes  & $>300$ & $>300$ & $298$ & $200$ & $70$ & $59$ & $25$ & $31$ \\
    \hline
  \end{tabular}
  \caption{Number of cutting planes considered in the outer approximation Algorithm \ref{alg:outer_approximation} as a function of the sample size $n$. For the smallest sample sizes $n$ the optimal solution could not be computed within the allocated maximum solution time.}
  \label{tab:cuts}
\end{table}

\subsection{The whole truth, and nothing but the truth}

In the previous section we demonstrated that exact sparse regression is marginally better in discovering all relevant features $A\%$. Nevertheless, the true advantage of exact sparse regression in comparison to heuristics such as \texttt{Lasso} is found to lie elsewhere. Indeed, both our method and the \texttt{Lasso} heuristic are with a sufficient amount of data eventually able to discover all relevant ($A\%\approx 100$) features.
In this section, we will investigate the ability of both methods to reject irrelevant features. Indeed in order for a method to tell the truth, it must not only tell the whole truth ($A\%\approx 100$), but nothing but the truth ($F\%\approx 0$). It is in the latter aspect that exact sparse regression truly shines.

\begin{figure}
\begin{center}
    \includegraphics[width=0.65\textwidth]{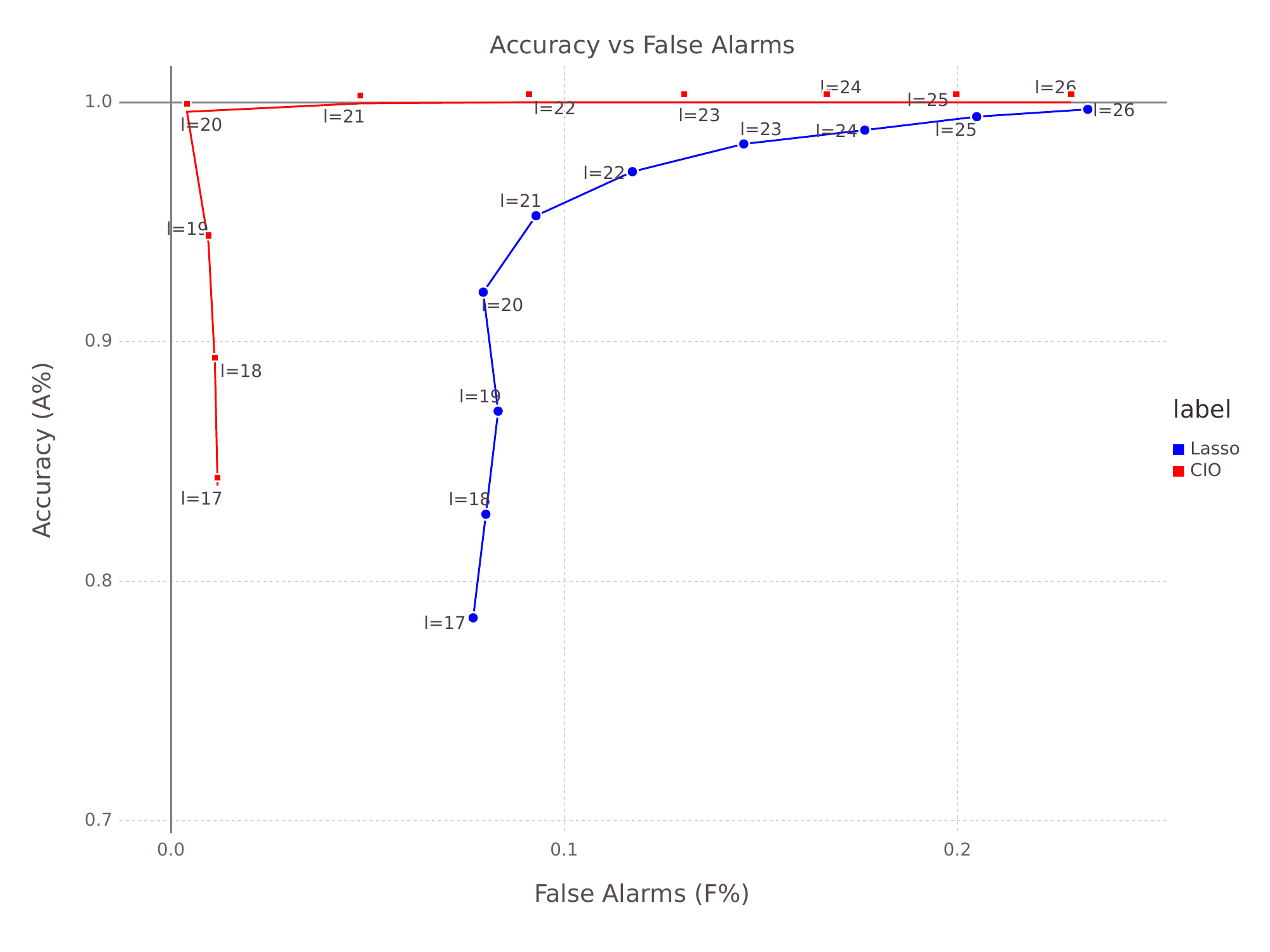}
\end{center}
\caption{The performance of exact sparse regression and the \texttt{Lasso} heuristic on synthetic data in terms of  accuracy $A\%$ and false alarms $F\%$.}
\label{fig:path}
\end{figure}

In Figure \ref{fig:path} we show the performance of the $(p, \ell)$-sparse regressors found with our exact method and the \texttt{Lasso} heuristic on the same synthetic data discussed in the previous section in terms of both their accuracy $A\%$ and false alarm rate $F\%$ in function of $\ell$. Again, the reported results are averages over 20 distinct synthetic data sets . Each of these data sets consisted of $n = 560$ observations. Among all potential third degree monomial features, again only $20$ where chosen to be relevant for explaining the observed data $Y$. Whereas in previous section, the true number of features was treated as a given, in practical problems $\ell$ must also be estimated from data. As we vary $\ell$ over the regression path $[f]$, we implicitly trade lower false alarm rates for higher accuracy. We have indeed a choice between including too many features in our model resulting in a high false alarm rate but hopefully discovering many relevant features, or limiting the number of features thus keeping the false alarm rate low but at the cost of missing features. One method is better than another when it makes this tradeoff better, i.e., obtains higher accuracy for a given false alarm rate or conversely a lower false alarm rate for the same accuracy. It is clear from the results shown in Figure \ref{fig:path} that exact sparse regression dominates the \texttt{Lasso} heuristic in terms of keeping a smaller false alarm rate while at the same time discovering more relevant monomial features.

Hence although both exact sparse regression and the \texttt{Lasso} heuristic are eventually capable to find all relevant monomial features, only the exact method finds a truly sparse regressor by  rejecting most irrelevant monomials from the obfuscating bulk. In practical situations where interpretability of the resulting regression model is key, the ability to reject irrelavant features can be a game changer. While all the results so far are demonstrated on synthetic data, we shall argue in the next section that also for real data sets similar encouraging conclusions can be drawn.

\subsection{Polynomial Input Ranking}

In the preceding discussions, the first step of our approach outlined in Figure \ref{fig:2-step-procedure} did not come into play as it was assumed that $p'=p$. Hence, no preselection of the inputs using the ranking heuristic took place. Because of the exponential number of regression features $f$ as a function of the input dimension $p$, such a direct approach may not be tractable when $p$ becomes large. In this part we will argue that the performance of the input ranking heuristic discussed in Section \ref{sec:ranking} is sufficient in identifying the relevant features while ignoring the obfuscating bulk.

To make our case, we consider synthetic data with hierarchical sparsity $(k, \ell)=(20, 40)$. That is, only 20 inputs in 40 relevant degree $r=3$ monomial features are relevant for the purpose of regression. In Table \ref{tab:heuristic} we report the average reduced input dimension $p'$ necessary for the input ranking  heuristic
to recover all relevant inputs. That is, all relevant $k$ inputs are among the top $p'$ ranked inputs.  For  example, for  $n=2,000$ and $p=1,000$ the input ranking heuristic  needs  to include $p'=286$  to cover all 20 true features, while for $n=10,000$ and $p=1,000$ the input ranking heuristic  needs  to include only $p'=78$. We note that as $n$ increases 
the  input ranking heuristic needs a smaller number of features  to identify the relevant ones. Note, however, that the false alarm rate 
of the input ranking heuristic remains high, that is, among the top $p'$ inputs many inputs were in fact irrelevant ($p'>k$). However, we do not aspire here to find all $k$ relevant inputs exactly, rather we hope to reduce the dimension to $p'$ without missing out any relevant inputs. Reducing the false alarm rate is done by means of exact hierarchical sparse regression in the second step of our overall algorithm.

\begin{table}
  \centering
  \begin{tabular}{|l|ccccc|}
    \hline
    Dimension $p'$ & $n=2\e{3}$ & $n=4\e{3}$ & $n=6\e{3}$ & $n=8\e{3}$ & $n=10\e{3}$\\
    \hline
    $p=200$ &72&41&36&54&27\\
    $p=400$ &117&109&75&80&62\\
    $p=600$ &222&166&128&68&104\\
    $p=800$ &361&230&96&95&102\\
    $p=1000$ &286&264&142&103&78\\
    \hline
  \end{tabular}
  \caption{Average size $p'$ necessary for our input ranking heuristic to identify all relevant features.}
  \label{tab:heuristic}
\end{table}

\subsection{Real Data Sets}
\label{ssec:real_data}
In the final part of the paper we report the results of the presented methods on several data sets found in the \texttt{UCI Machine Learning Repository} found under \texttt{https://archive.ics.uci.edu/ml/datasets.html}. Each of the datasets was folded ten times into $80\%$ training data, $10\%$ validation data and $10\%$ test data $\mc T$. No preprocessing was performed on the data besides rescaling the inputs to have a unit norm. 

We report the prediction performance of four regression methods on each of the test data sets. The first regression method we consider is ordinary Ridge regression. Ridge regression allows us to find out whether considering nonlinear regressors has merit. The second method we consider is polynomial kernel regression with degree $r > 1$ polynomials as in \eqref{eq:abstract_regression}. This nonlinear but non-sparse method on its part will allow us to find out the merits of sparsity for the purpose of out-of-sample prediction performance. The third method is the $\ell_1$-heuristic described before and which will allow us to see whether exact sparse formulations bring any benefits. We used the input ranking algorithm described in Section \ref{sec:ranking} to limit the number of potentially relevant inputs to at most $p'=20$. Each of these methods is compared to our hierarchical exact regressor in terms of out-of-sample test error
\[
  \mathrm{TE} \defn \frac{\sum_{t\in \mc T}\norm{y_t-h(x_t)}_2}{\sqrt{\abs{\mc T}}}.
\]
The hyperparameters of each method were chosen as those best performing on the validation data from among $k \in [p']$, $\ell \in [100]$ and polynomial degree ranging in $r\in [4]$. The average out-of-sample performance on the test data and sparsity of the obtained regressors using the \texttt{Lasso} heuristic and exact hierarchical sparse regression is shown in Table \ref{table:datasets}.

\begin{table}
  \centering
  \begin{tabular}{|lcc|c|cc|ccc|cccc|}
    \hline
  \multicolumn{3}{|c|}{} & RR & \multicolumn{2}{c|}{SVM} & \multicolumn{3}{c|}{PL} & \multicolumn{4}{c|}{CIO}\\
  \hline
  Problem & $n$ & $p$   & TE & $r^\star$   &  TE  &  $r^\star$ & $\ell^\star$ & TE & $r^\star$ & $k^\star$ & $\ell^\star$ & TE \\
  \hline
  \citeauthor{brooks1989airfoil} & 1502 & 5 & 4.62 & 3 & 4.04 & 4 & 65 & 3.61 & 4 & 5 & 100 & \textbf{3.53} \\
  \citeauthor{yeh1998concrete} & 1029 & 8 & 10.2 & 3 & \textbf{6.00} & 4 & 94 & 6.11 & 3 & 8 & 100 & 6.27 \\
  \citeauthor{tsanas2012energy} I& 768 & 8 & 2.79 & 4 & 2.40 & 4 & 49 & 2.52 & 4 & 8 & 100 & \textbf{2.37} \\
  \citeauthor{tsanas2012energy} II& 768 & 8 & 3.12 & 4 & \textbf{1.72} & 4 & 50 & 2.97 & 4 & 6 & 70 & 2.79 \\
  \citeauthor{cortez2009wine} & 1599 & 11 & 0.65 & 2 & 0.70 & 4 & 70 & 0.69 & 3 & 7 & 95 & \textbf{0.63} \\
  \citeauthor{cortez2009wine} & 4898 & 11 & 0.74 & 2 & 0.72 & 3 & 95 & 0.70 & 3 & 11 & 95 & \textbf{0.70} \\
  \citeauthor{zhou2014music} I & 1058 & 68 & \textbf{16.4} & 2 & 20.3 & 2 & 30 & 16.5 & 4 & 18 & 65 & 18.5 \\
    \citeauthor{zhou2014music} II & 1058 & 68 & \textbf{44.0} & 2 & 50.4 & 2 & 37 & 44.8 & 2 & 20 & 25 & 45.8 \\
  \hline
  \end{tabular}
  \caption{Out-of-sample performance of Ridge regression (RR), polynomial kernel regression (SVM), polynomial lasso (PL) and exact hierarchical regression (CIO) on several UCI Data sets. We give the optimal hyper parameters for each of these methods as well.}
  \label{table:datasets}
\end{table}

As one could expect, considering nonlinear monomial features is not beneficial for prediction in all data sets. Indeed, in two data sets ordinary Ridge regression provides the best out-of-sample performance. A similar remark can be made with regards to sparsity. That is, for three data sets adding sparsity does not immediately yield any benefits in terms of prediction power. Nevertheless, in those situations were nonlinear and sparse regression is beneficial the results again point out that there is a benefit to exact sparse regression rather than heuristic approaches.

\section{Conclusions}

We discussed a scalable hierarchical sparse regression method based on a smart heuristic and modern integer optimization for nonlinear regression. We consider as the best regressor that degree $r$ polynomial of the input data which depends on at most $k$ inputs counting at most $\ell$ monomial terms which minimizes the sum of squares prediction errors with a Tikhonov loss. This hierarchical sparse specification aligns well with big data settings where many inputs are not relevant for prediction purposes and the functional complexity of the regressor needs to be controlled to avoid overfitting. Using a modern cutting plane algorithm, we can use exact sparse regression on regression problems of practical size. The ability of our method to identify all $k$ relevant inputs as well as all $\ell$ relevant monomial terms and reject all others was shown empirically to experience a phase transition. In the regime where our method is statistically powerful, the computational complexity of exact hierarchical regression was empirically on par with \texttt{Lasso} based heuristics taking away their main propelling justification. We have empirically shown that we can outperform heuristic methods in both finding all relevant nonlinearities as well as rejecting obfuscating ones. 

\section*{Acknowledgements}

The second author is generously supported by the Early Postdoc.Mobility fellowship P2EZP2 165226 of the Swiss National Science Foundation.

\setlength\bibitemsep{10pt}
\printbibliography

\end{document}